\documentclass[12pt]{amsart}
\usepackage[headings]{fullpage}

\usepackage{amsmath,amssymb,amsthm}
\usepackage{color}
\usepackage{graphicx}
\usepackage{enumerate}
\usepackage{eufrak}
\usepackage{url}
\usepackage{slashed}
\usepackage{bm}
\usepackage{tikz-cd}
\usepackage{pb-diagram}
\usepackage[english]{babel}

\usepackage{enumitem}
\usepackage[bookmarks=true,%
    colorlinks=true,%
    linkcolor=blue,%
    citecolor=blue,%
    filecolor=blue,%
    menucolor=blue,%
    urlcolor=blue,%
    breaklinks=true]{hyperref}

\usepackage[all]{xy}
\usepackage{verbatim}

\newcommand{\HP}{\mathbb{H}}

\newcommand{\Q}{\mathbb{Q}}
\newcommand{\R}{\mathbb{R}}
\newcommand{\Z}{\mathbb{Z}}

\newcommand{\Arc}{\mathrm{Arc}}

\newcommand{\Aut}{\mathrm{Aut}}

\newcommand{\Card}{\mathrm{Card\,}}
\newcommand{\Cay}{\mathrm{Cay}}

\newcommand{\End}{\mathrm{End}}

\newcommand{\Fill}{\mathrm{Fill}}

\newcommand{\GL}{{\mathrm {GL}}}

\newcommand{\Ker}{\mathrm{Ker}}

\newcommand{\Max}{\mathrm{Max}}
\newcommand{\Min}{\mathrm{Min}}

\newcommand{\PGL}{{\mathrm {PGL}}}

\newcommand{\Syst}{{\mathrm {Syst}}}

\newcommand{\Sys}{{\mathrm {Sys}}\,}

\newcommand{\arcosh}{\mathrm{arcosh}}

\renewcommand{\mod}{\,\mathrm{mod}\,}

\newtheorem{thm}{Theorem}
\newtheorem{cor}[thm]{Corollary}
\newtheorem{crit}[thm]{Criterion}
\newtheorem{lemma}[thm]{Lemma}
\newtheorem{prop}[thm]{Proposition}

\newtheorem*{thm*}{Theorem}
\newtheorem*{thm12}{Theorem12}

\newtheorem*{thm25}{Theorem 25}

\newtheorem*{crit*}{Criterion}
\newtheorem*{crit18}{Criterion 18}

\newtheorem*{prop22}{Proposition 22}

\newcommand{\param}{{\mathchoice{\mkern1mu\mbox{\raise2.2pt\hbox{$
\centerdot$}}
\mkern1mu}{\mkern1mu\mbox{\raise2.2pt\hbox{$\centerdot$}}\mkern1mu}{
\mkern1.5mu\centerdot\mkern1.5mu}{\mkern1.5mu\centerdot\mkern1.5mu}}}

\renewcommand{\setminus}{{\smallsetminus}}

\begin{document}
\title{Small  Systole Sets and Coxeter Groups}
\author{Ingrid Irmer}
\author{Olivier Mathieu}
\address{Department of Mathematics\\
Southern University of Science and Technology\\Shenzhen, China
}
\address{SUSTech International Center for Mathematics\\
Southern University of Science and Technology\\Shenzhen, China
}
\address{Institut  Camille Jordan du CNRS\\
UCBL, 69622 Villeurbanne Cedex, France
}

\email{ingridmary@sustech.edu.cn}
\email{mathieu@math.univ-lyon1.fr}

\today

\begin{abstract} The systoles of a  hyperbolic surface $\Sigma$ are the shortest closed geodesics. We say that the  systoles {\it fill} the surface if the set $\Syst(\Sigma)$ of all systoles cuts $\Sigma$ into polygons. 

We refine an idea of Schmutz \cite{SS} to construct 
closed hyperbolic surfaces $\Sigma$ of arbitrarily large genus with a  small set $\Syst(\Sigma)$ that fills. In fact, for the surfaces $\Sigma$ considered, the cardinality of $\Syst(\Sigma)$ is in $o(g/\sqrt{\ln\,g})$, where $g$ is the genus of $\Sigma$.
The proof is based on the theory Coxeter groups, combined with some elementary number theory. 
 
\end{abstract}

\maketitle

{\footnotesize
\tableofcontents
}


\section{Introduction}
\label{sect:intro}

\noindent {\it 1.1 General introduction}

\noindent In this paper, surfaces are assumed to be oriented. A {\it systole} of a hyperbolic surface 
$\Sigma$ is an essential closed geodesic of minimal length. Let $\Syst(\Sigma)$ be the set of all systoles of $\Sigma$. We say that $\Syst(\Sigma)$ {\it fills} the surface if it cuts $\Sigma$ into polygons.

For $g\geq 2$,  the Teichm\"uller space  ${\mathcal T}_g$ is a manifold used for parameterising closed hyperbolic surfaces of genus $g$. In \cite{T85}, Thurston defined a remarkable subspace $\mathcal{P}_{g}\subset {\mathcal T}_g$, called the {\it Thurston spine}. It consists of all surfaces $\Sigma\in{\mathcal T}_g$ for which $\Syst(\Sigma)$ fills.

Since  $\mathcal{P}_{g}$  is  nonempty by \cite{T85}, one can meaningfully define the integer $\Fill(g)$ as the smallest cardinality of $\Syst(\Sigma)$ when $\Sigma$ varies over $\mathcal{P}_{g}$. Trying to understand the dimension of $\mathcal{P}_{g}$ leads to the question of finding an upper bound for $\Fill(g)$. In this direction, we prove

\begin{thm25} There exists an infinite set $A$ of
integers $g\geq 2$ such that

$$\Fill(g)\leq
{57\over \sqrt{\ln\ln\ln g}}\,\,
{g\over \sqrt{ \ln g}}$$

\noindent for any $g\in A$.
\end{thm25}

In fact, the theorem 25 proved in the main body is  slightly stronger than the previous statement. 
Since $57/\sqrt{\ln\ln\ln g}$ belongs to $o(1)$, it implies the result stated in the abstract. 
From our main result, we will deduce, in a forthcoming preprint \cite{IM}, a related bound for the codimension of the Thurston spine $\mathcal{P}_{g}$.

\bigskip\noindent
{\it 1.2 Previous works
for $\Fill(g)$}

\noindent 
The idea of studying examples for which the systoles cuts the surface into regular right-angled polygons goes back to \cite{SS}, and this idea has been used extensively in the subsequent works
\cite{APP}\cite{S}\cite{FB}.
Schmutz's paper \cite{SS} seems to have been motivated by the observation that many critical points of mapping class group-equivariant Morse functions are  of this type. Classical examples are the Bolza surface and the Klein quartic.
This investigation also led to the study of upper bounds of $\Fill(g)$. Hyperbolic surfaces with $2g$ systoles have been found in \cite{SS},\cite{APP} and \cite{S}. The recent  result of \cite{FB} can be reformulated as follows.

\begin{thm*}\cite{FB}
There is infinite set $B$ of integers $g\geq 2$ 
and an  increasing function $\psi$
(discussed below)  such that 

\centerline{$\Fill(g)\leq g/\psi(g)$ for any  $g\in B$.} 
\end{thm*} 

The function $\psi$ is only implicitly defined in {\it loc. cit.}, but a rough estimate is  given in a footnote. It is clear that $\psi(g)$ is in $o(\lg^*g)$, where  $\lg^*$ is, essentially, the inverse of the Ackerman function $n\mapsto F(4,n)$ (see \cite{CLRS}, Section 3.2, pp. 58-59). In particular $\lg^*$ is smaller than the $m$th-iterate $\lg^{(m)}=\lg\circ\lg\circ \dots\lg$, of the base-$2$ logarithmic function $\lg$, for any $m\geq 1$. Therefore $\psi(g)$ is much smaller than the factor $\sqrt {\ln g}$ in our denominator.

Conversely, a rough lower bound $\sim\pi g/\ln g$ has been found in \cite{APP}.  According to {\it loc.cit.}, this lower bound seems difficult to obtain. Intuitively, the difficulty comes from the fact that a small number $N$ of filling systoles implies a relatively large systole length, at least $O(g/N)$. However, as a loose general rule, the number of systoles increases with the systole length.

\bigskip\noindent
\noindent\textit{ 1.3
Main ideas of the proof and organisation of the paper}

\noindent
The surfaces studied in this paper were motivated by the examples in Theorem 36 of \cite{SS}. In the examples of {\it loc. cit.} the systoles cut the surface into regular right-angled polygons.
However, here we will use the refined notion of {\it standard} tesselations.  In order to explain this, we first need to give some definitions.

For simplicity, {\it we will only consider hexagonal tesselations}. A {\it decoration} of the regular right-angled hexagon $P$ is a cyclic indexing of the edges by $\Z/6\Z$. Since a cyclic indexing of the edges
defines an orientation of the hexagon,
$P$ admits exactly two decorations, up to orientation preserving isometry.
By definition, a {\it standard} tesselation 
 of an oriented hyperbolic closed surface $\Sigma$ 
 is a tesselation by {\it decorated} regular right-angled hexagons.
 By definition, the {\it curves} of a standard tesselation $\tau$ are the maximal geodesic components of the 
 the $1$-skeleton $\tau_1$ of $\tau$. The standard tesselation $\tau$ is called {\it $2k$-regular} if all its curves consist of exactly $2k$ edges. Since each side of a regular right-angled hexagon has length 
 $\arcosh\, 2$, all curves of a $2k$-regular standard tesselation are closed geodesics of length
 $2k\,\arcosh 2$. We also define the Coxeter group
 $W(k)$ by the presentation
 
 $$\langle (s_i)_{i\in \Z/6\Z} \mid
s_i^2=1, \, (s_is_{i+1})^2=1,
\,\mathrm{and}\,(s_is_{i+2})^k=1,
\,\forall i\in\Z/6\Z 
\rangle$$

\noindent Let $\epsilon:W(k)\to\Z/2\Z$ be the sign homomorphism, defined by 
$\epsilon(s_i)=-1$, for any $i\in\Z/6\Z$, and set
$W(k)^+=\Ker\,\epsilon$.

Before going to the core of the proof, we will explain the connection between the Coxeter group $W(k)$ and the use of decorated tiles.
The tesselations considered here are, somehow, ``doubly 
regular": the tiles are regular hexagons, and all curves have the same length. These two properties
appear in \cite{SS} and the subsequent papers \cite{S}\cite{FB}. 
The advantage of adding a decoration to the tiles
is explained by the following
observation, proved in Section 3.

\begin{thm12} 
 There is a one-to-one correspondence between 
 \begin{enumerate}
 
 \item the $2k$-regular standard tesselations $\tau$ of closed oriented surfaces, and

\item the finite index subgroups $H\subset W(k)^+$ 
satisfying (\ref{condition}.1) and
(\ref{condition}.2).

\end{enumerate}
\end{thm12}

\noindent The assertions (\ref{condition}.1) and
(\ref{condition}.2) are described in Lemma 11 of
Section 3. These conditions, which are usually satisfied, are easy to check.  

We can now describe the main ideas of the proof. Using the previous statement, obtaining a bound on $\Fill(g)$ is  reduced to the theory of  Coxeter groups. 
The theory of right-angled Coxeter groups is a classical tool to investigate the tesselations of the Poincare half-plane $\HP$ \cite{Da}. Our construction is similar, but the Coxeter groups $W(k)$ are not right-angled.
The proof contains three parts.

In the first part, namely in Sections 2 and 3, we look at the delicate question - are the set of curves of a $2k$-regular standard tesselation $\tau$  and  the  set of systoles of the corresponding surface  identical? A partial answer  is provided at the end of Section 3.

\begin{crit18} Assume that $k\geq 4$ is even. Let $H$ be a subgroup of $W(k)^+$, any conjugate of which
 intersects $B_{4k}$ trivially, where $B_{4k}$ is the ball of radius $4k$ in $W(k)$. Then the set  of curves of the corresponding tesselation $\tau$ is the set of systoles.
\end{crit18}

The proof of the criterion is quite long,  and it
uses a new result on the  combinatorics of Coxeter groups $W$, namely the Theorem 9 proved in Section 2.  We define the {\it Cayley complex} 
$\Cay^+\,W$ of  $W$ by attaching some $2$-cells to its Cayley graph $\Cay\,W$. Theorem 9  involves Coxeter groups endowed with right-angled partition. It
shows that some``relatively short'' loops of  $\Cay\,W$ are null-homotopic in $\Cay^+\,W$. 

In the second part of the proof, i.e. Section 4, we find an  upper bound for the index of a subgroup $H$ satisfying Criterion 18:

\begin{prop22}  For any $k\geq 3$, there exists a normal subgroup
$H(k)$ of $W(k)$ satisfying the criterion 18 with

\centerline{
$[W(k):H(k)]\leq 3^{72k\phi(2k)}$,}

\noindent where $\phi$ is Euler totient function.
\end{prop22}

Its proof uses the Tits representation  $\rho:W(k)\to \GL_6(K)$ \cite{T61}, where $K$ is the number field 
$\Q(\cos\pi/k)$. We have $H(k)=\{w\in W(k)\mid \rho(w)\in \Gamma\}$, where $\Gamma$ is a suitable congruence subgroup of $\GL_6(K)$.

The last part of the proof, in Section 5, explains the factor ${57\over \sqrt{\ln\ln\ln g}}$. It is based on Landau's Theorem \cite{L}\cite{HW} about  the asymptotics of $\phi(k)/k$, which  is a corollary of the classical prime number theorem
\cite{VP}\cite{H}.

\section{The 2-dimensional Cayley Complex $\Cay^+\,W$}
\label{sect1}
Given a Coxter system $(W,S)$, the Tits combinatorics \cite{T68} describes the loops in its Cayley graph $\Cay\, W$.

 In this section, we define a Cayley complex $\Cay^+\,W$ obtained by attaching a collection of 2-cells to $\Cay\, W$, for each commutative rank two parabolic subgroup of $W$. This square complex
$\Cay^+\,W$ is unrelated with the well-known simplicial complexes of \cite{Da}, like the Coxeter's complex and the Davis's complex.

We also define the notion of right-angled partition of a Coxeter group $W$. Theorem \ref{null}, proved in this section, states that for a Coxeter group $W$ endowed with a right-angled partition, the ``relatively short" loops of $\Cay\, W$ are null-homotopic in $\Cay^+\,W$. This result will be the main ingredient of the proof of Criterion \ref{systole} in Section 3.

\bigskip\noindent
{\it 2.1 Coxeter Groups}

\noindent Let $S$ be a set. A square matrix $M=(m_{s,t})_{s,t\in S}$ is a {\it Coxeter matrix} if it satisfies the following conditions:
\begin{enumerate}
\item $m_{s,s}=1$, 

\item  for $s\neq t$, $m_{s,t}$ belongs to $\Z_{\geq 2}\cup \{\infty\}$,

\item $m_{s,t}=m_{t,s}$.
\end{enumerate}
\noindent The group $W$ defined by the presentation

$$\langle s\in S\mid (st)^{m_{s,t}}=1, \forall s,t\in S \,{\mathrm {with}}\,m_{s,t}\neq\infty\rangle$$

\noindent is called the {\it Coxeter group} associated with the Coxeter matrix $M$. Unless stated otherwise, it will be assumed that the set $S$ is finite. Its cardinality is called the {\it rank} of $W$. The pair $(W,S)$ is called a {\it Coxeter system}. 

Let $\epsilon:W\to\{\pm 1\}$ be the group homomorphism uniquely defined by the property that $\epsilon(s)=-1$ for any $s\in S$. This is called the {\it signature homorphism}.

Denote by $\mathcal{W}_S$ the free monoid generated by $S$. An  element $w$ of $\mathcal{W}_S$ is a word
$w=s_1\dots s_n$ where $s_1,\dots,s_n$ belong to the alphabet $S$. By definition $l(w):=n$ is the length of $w$. There is a natural monoid homomorphism $\mathcal{W}_S\to W, w\mapsto \overline{w}$ whose restriction to $S$ is  the identity. The element $w$ is called a {\it word representative} of $\overline w$.
The {\it Bruhat length} of an element $u\in W$, 
denoted by $l(u)$,
is the minimal length of a word representative of 
$u$.  A word $w\in\mathcal{W}_S$ is called {\it reduced}
if $l(w)=l({\overline w})$.

\bigskip\noindent
{\it 2.2 The Tits word combinatorics}

\noindent
For any distinct $s,t\in S$ with
$m_{s,t}<\infty$, let $w(s,t)$ be the
word of length $m_{s,t}$ starting with $s$ and alternating the letters $s$ and $t$. The relation
$(st)^{m_{s,t}}=1$ is in fact equivalent to

\centerline{$\overline{w(s,t)}=\overline{w(t,s)}$.}

The {\it subwords} of a word $w=s_1 s_2\dots s_n\in \mathcal{W}_S$ are the words
$s_{i_1}s_{i_2}\dots s_{i_k}$, where
$1\leq i_1<i_2\dots<i_k\leq n$ and the 
{\it substrings} of $w$ are the subwords 
$s_ps_{p+1}\dots s_q$ where 
$1\leq p\leq q\leq n$. 

An {\it elementary reduction} \cite{Ca}
is a pair of words $(w,w')$ such that 
$w'$ can be obtained from $w$ by one of the following
reductions:
\begin{itemize}
\item{\textit{Reduction of first type}: deleting two consecutive identical letters
in $w$, or}
\item{\textit{Reduction of second type}: replacing in $w$ a substring $w(s,t)$  by $w(t,s)$.}
\end{itemize}
\noindent 
Consequently,  we have   $l(w')= l(w)-2$ for a reduction of the fist type, and $l(w')= l(w)$ otherwise. 

\begin{thm} (Tits)\label{T68} Let $w,w'$ be two words with $\overline{w}=\overline{w}'$. If $w'$ is reduced,
one can transform $w$ into $w'$ by a sequence of elementary reductions.
\end{thm}

\noindent Besides the original reference in French \cite{T68}, the reader can consult Davis's book \cite{Da}, section 3.4. (The elementary reductions are called elementary $M$-operations in {\it loc.cit.}.)  Our presentation of the Tits Theorem is  close to Casselman's webpage \cite{Ca}.  

\bigskip\noindent
{\it 2.3 Girth of $W$}

\noindent Set $\gamma(W)=2\,\Min_{s\neq t} \, m_{s,t}$.
The integer $\gamma(W)$, possibly infinite, is the girth of the Cayley graph of $W$, see  \cite{Lu}
Lemma 2.1. It will also be  called the {\it girth} of $W$. 
For any distinct $s,t\in S$, we have 
$l\bigl(w(s,t)\bigr)\geq \gamma(W)/2$. Therefore any
word $w$ with $l(w)<\gamma(W)/2$ can be reduced only by
reductions of the first type. A consequence of 
Theorem \ref{T68} is the following

\begin{cor}\label {girth}
Let $w,w'$ be two words with $\overline{w}=\overline{w}'$. Assume that  $l(w)<\gamma(W)/2$. 

\begin{enumerate}
\item if $w$ and $w'$ are reduced,  we have $w=w'$,

\item if $w$ is not reduced, it contains a substring $ss$ for some $s\in S$.
\end{enumerate}
\end{cor}

\bigskip\noindent
{\it 2.4 Loops} 

\noindent By definition, a {\it cyclic word} on the alphabet $S$ of length $n$ is a word $w=s_1\dots s_n$ in $\mathcal{W}_S$ modulo a cyclic permutation. For example the cyclic words $s_1s_2s_3$ and $s_3s_1s_2$ are  equal.

For a cyclic word $w=s_1\dots s_n$, it will be convenient to assume that the indices belong to $\Z/n\Z$. The {\it substrings} of length $l\leq n$ of the cyclic word $w=s_1s_2\dots s_n$ are the words $u=s_is_{i+1}\dots s_{i+l-1}$, for some 
$i\in \Z/n\Z$. For example $s_4s_1$ is a substring of the cyclic word $s_1s_2s_3s_4$.

A word $w=s_1\dots s_n$ in $\mathcal{W}_S$ of length $n>0$ is called a {\it loop} if $\overline{w}=1$. For a cyclic word $w=s_1\dots s_n$, the condition $\overline{w}=1$ is independent of its representatives in $\mathcal{W}_S$. It follows that 
$w$ is called a {\it cyclic loop} if any of its representatives in $\mathcal{W}_S$ is a loop. Since $\epsilon(s)=-1$ for any $s\in S$, the length of any loop or cyclic loop is even. 

\begin{cor}\label{cyclicgirth} Let $w=s_1\dots s_{2n}$ be a cyclic loop, with $2n<\gamma(W)$. Then
there are two distinct indices $i,j\in\Z/2n\Z$ such that $s_i=s_{i+1}$ and $s_j=s_{j+1}$. 
\end{cor}

\begin{proof} Set $u=s_1s_2\dots s_n$ and $v=s_{2n}s_{2n-1}s_{n+1}$. We have $\overline{u}=\overline{v}$ and $l(u)=l(v)=n<\gamma(W)/2$. It follows that $u$ is reduced iff  $v$ is reduced.

If both $u$ and $v$ are reduced, it follows from the first assertion of Corollary \ref{girth} that $u=v$, hence we have $s_n=s_{n+1}$ and    $s_{2n}=s_{2n+1}$.

Otherwise, we can transform $u$ and $v$ into  reduced words by a sequence of elementary reductions of the first type. By the second assertion of Corollary \ref{girth},  there exist $i,j$ with $1\leq i<n$ and $n\leq j<2n$ such that $s_i=s_{i+1}$ and $s_j=s_{j+1}$. 

In both cases, the corollary is proved.
\end{proof}

\bigskip\noindent
{\it 2.5 Parabolic subgroups} 

\noindent For a subset $I$ of $S$, let $W_I\subset W$ be the subgroup  generated by $I$, let 
$\mathcal{W}_I$ be the set of words on the alphabet $I$. The subgroups $W_I$ are called the {\it parabolic subgroups} of $W$.  It is well-known that $(W_I,I)$ is a Coxeter system, with Coxeter matrix $(m_{s,t})_{s,t\in I}$, see e.g. \cite{Da} Section 4.1.

It is clear from  Theorem \ref{T68} that any reduced expression of an element $w\in W_I$ is in $\mathcal{W}_I$. It follows that $W_I\cap W_J=W_{I\cap J}$ for any two subsets $I,J$ of $S$. Given $w\in W$, the smallest subset $I$ such that $w\in W_I$ is called the {\it support} of $w$.

For a subset $I$ in $S$ and $t\in S$, set 

\centerline{$I(t):=\{s\in I\mid m_{s,t}=2\}
=\{s\in I\mid s\neq t \,\mathrm{and}\, st=ts\}$, and}

\centerline{$t^{W_I}=\{t^w\mid w\in W_I\}$,}

\noindent where, as usual, $t^w:=wtw^{-1}$.

\begin{lemma}\label{parabolic} Let $I$ be a subset of $S$. 
\begin{enumerate}
\item For any $t\in S\setminus I$,  the centraliser of $t$ in $W_I$ is $W_{I(t)}$. In particular, we have $t^{W_I}\simeq W_I/W_{I(t)}$.

\item For elements $t\neq t'$ of $S\setminus I$, we have $t^{W_I}\cap t'^{W_I}=\emptyset$.
\end{enumerate}
\end{lemma}

This lemma appears to be well-known. Since we did not find an exact reference, we provide a quick proof.

\bigskip
\noindent{\it Proof of Assertion (1).} Set 

\centerline{$W_I^{I(t)}=\{w\in W_I\mid l(ws)>l(w) \,\forall s\in I(t)\}$.}

\noindent Let $w\in W_I^{I(t)}$ with $w\neq 1$ and let $w=s_1\dots s_n$ be any reduced expression for $w$. Since $t$ is not in the support of $w$, the word $s_1\dots s_nt$ is reduced. Since $s_n\notin I(t)$, no reduction of second type involves $t$. Therefore any reduced expression of $wt$ ends with the letter $t$, therefore 

\centerline{$wt\neq tw$.}

By Section 4.5 of \cite {Da},  any  $w\in W_I$ can be uniquely written as $w=uv$, where $u\in W_I^{I(t)}$ and
$u\in W_{I(t)}$. Thus the previous statement is equivalent to Assertion (1).

\bigskip
\noindent{\it Proof of Assertion (2)} Let $t\neq t'$ be elements in $S\setminus I$. The support of any element in  $w\in t^{W_I}$ (resp. $w\in t'^{W_I}$) contains $t$ but not $t'$ (resp. contains $t'$ but not $t$). Therefore $t^{W_I}$ and $t'^{W_I}$ are disjoint.
\qed

\bigskip\noindent
{\it 2.6 The  Cayley complex  $\Cay^+\,W$}

\noindent By definition, the {\it Cayley graph} of $W$, denoted $\Cay\, W$, is the graph whose vertices are the elements $v\in W$ and whose edges are the pairs $(v,vs)$, for $v\in W$ and $s\in S$. We will now define the {\it Cayley complex} 
$\Cay^+\,W$ by attaching some 2-cells to $\Cay\, W$.

 Let $\mathfrak{P}$ be the set of pairs  $I=\{s,t\}$ of commuting elements of $S$.  For $I\in \mathfrak{P}$, any
 $W_I$-coset $v W_I$ is a subgraph of $\Cay\,W$ consisting of a $4$-cycle,  which can be seen as the boundary of a plain 
 square ${\bf c}(v,I)$. An example is the square
 ${\bf c}_1$ shown in Figure \ref{Figure1SystoleStory}. Therefore we can attach the $2$-cell  ${\bf c}(v,I)$ along $\partial{\bf c}(v,I)$ to  the Cayley graph $\Cay\,W$. By definition, the {\it Cayley complex} $\Cay^+\,W$ is the $2$-dimensional complex obtained by attaching the $2$-cells  ${\bf c}(v,I)$, where $I$ varies over $\mathfrak{P}$ and $v$ varies over a set of representatives of $W/W_I$.  
 
It remains to add one remark to complete the definition of $\Cay^+\,W$. The group $W_I$ acts (by the right action) of the coset $v W_I$, and this action can be extended to the cell ${\bf c}(v,I)$. The two generators $s$ and $t$ of $W_I$ are  the median reflections of the square ${\bf c}(v,I)$. We require that the 2-cells ${\bf c}(v,I)$ are glued compatibly with the $W_I$-action. It follows that  the $W$-action on $\Cay\,W$ extends naturally to $\Cay^+\,W$.
 
Note that any word $w=s_1s_2\dots s_n$ in $\mathcal{W}_S$ defines a path, denoted by $\vert w\vert$, in the Cayley graph $\Cay\,W$. The $n$ successive edges of $\vert w\vert$ are

\centerline{ $(1,s_1), (s_1,s_1s_2)\dots (s_1s_2\dots s_{n-1},s_1s_2\dots s_n)$.}

Let $v\in W$. By definition the path $\vert w\vert$ is based at $1$ and the path $v.\vert w\vert$ is based at $v$. 
By $W$-equivariance, it is clear that  

\begin{itemize}
\item for any $w_1,w_2\in \mathcal{W}_S$, the paths $\vert w_1\vert$ and $\vert w_2\vert$ are homotopic in $\Cay^+\, W$ iff $v.\vert w_1\vert$ and $v.\vert w_2\vert$ are homotopic, and

\item for any loop $w\in\mathcal{W}_S$, $\vert w\vert$  is null-homotopic in $\Cay^+\, W$ iff $v.\vert w\vert$ is. 
\end{itemize}

\noindent Therefore the next statement and its proof only involve loops based at $1$.

For $s\in S$ and $I\subset S$, set

\centerline{$I(s)=\{t\in I\mid m_{s,t}=2\}$.}

\begin{lemma}\label{homotopic} Let $u\in \mathcal{W}_{I(s)}$ be a reduced word. Then the paths $\vert sus\vert$ and $\vert u\vert$ are homotopic in  $\Cay^+\,W$.
\end{lemma}

\begin{proof} By definition, we have $u=t_1\dots t_n$ where all $t_i$ commute with $s$ and $n=l(u)$. For each integer $i$ with $1\leq i\leq n$, let  $I_i=\{s,t_i\}$ and ${\bf c_i}={\bf c}(t_1t_2\dots t_{i-1},I_i)$. Then ${\bf c_1}\cup{\bf c_2}\dots\cup {\bf c_n}$ is a rectangle of $\Cay^+\,W$. As shown in Figure \ref{Figure1SystoleStory}, the lower side of the rectangle is the path $\vert t_1t_2\dots t_n\vert$ and the three other sides represent the path $\vert s t_1t_2\dots t_n s\vert$. It follows that inside $\Cay^+\,W$, the path $\vert sus\vert$ and $\vert u\vert$ are homotopic.
\end{proof}


\begin{figure}
\centering
\includegraphics[width=0.8\textwidth]{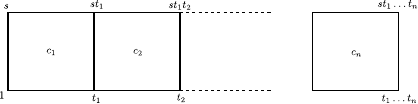}
\caption{The rectangle ${\bf c_1}\cup{\bf c_2}\dots\cup {\bf c_n}$ in $\Cay^+\,W$.}
\label{Figure1SystoleStory}
\end{figure}

\bigskip\noindent
{\it 2.7 Gal's Theorem}  
 
\noindent Following \cite{Ga}, a partition $S=R\sqcup B$ of $S$ is called a {\it Gal's partition} if
$m_{s,t}$ is an even integer or is infinite for any $s\in R$ and $t\in B$. For a Gal's partition $S=R\sqcup B$, there are group homorphisms $\mu_R: W\to W_R$ and
$\mu_B:W\to W_B$ uniquely defined by 

\centerline{$\mu_R(s)=s$ if $s\in R$ and $\mu_R(s)=1$ if $s\not\in R$,}

\centerline{$\mu_B(s)=1$ if $s\in B$ and $\mu_B(s)=s$ if $s\not\in B$.}

\noindent Set

\centerline{$\Ker \mu_R:=\overline{W}_B$.}

\noindent The notation $\overline{W}_B$ is intended to emphasise that  $\Ker \mu_R$ is the normal closure of $W_B$
\cite{Ga}. Set

\centerline{$\overline B=\cup_{t\in B}\,t^{W_R}$.}
 
\begin{thm}\label{Gal} (Gal) Let $S=R\sqcup B$ be a Gal's partition.

Then the pair $(\overline{W}_B,\overline{B})$ is a Coxeter system, possibly of infinite rank. Moreover, its Coxeter matrix
$(m_{\sigma,\tau})_{\sigma,\tau\in \overline{B}}$ is defined by

\begin{enumerate}
\item If $\sigma=s^w$ and $\tau=t^w$ for some $s,t\in B$ and $w\in W_R$, then  $m_{\sigma,\tau}=m_{s,t}$,

\item If $\sigma=t^w$ and $\tau=t^{ws}$ for some $t\in B$, $s\in R$ and $w\in W_R$, then  $m_{\sigma,\tau}=m_{s,t}/2$, and

\item otherwise, we have $m_{\sigma,\tau}=\infty$.
\end{enumerate}
\end{thm}
 
For the proof, see Proposition 2.1 and Corollary 3.1 of \cite{Ga}. 

\bigskip\noindent
{\it 2.8 Right-angled partitions}  

\noindent In order to use Gal's Theorem, we will restrict ourselves to a certain type of Gal's partitions. Recall that a Coxeter group is called {\it right-angled} if the non-diagonal entries of its Coxeter matrix are $2$ or $\infty$. By analogy, a partition $S=R\sqcup B$ of $S$  will be called a {\it right-angled} partition if $m_{s,t}=2$ or $\infty$, for any $s\in R$ and $t\in B$. 

\begin{lemma}\label{bargirth} Let $S=R\sqcup B$ be a right-angled  partition. Then we have

\centerline{$\gamma(\overline{W}_B)=
\gamma(W_B)$.}
\end{lemma}

\begin{proof} Let $\sigma,\tau\in\overline B$. If $\sigma=t^w$,  $\tau=t^{ws}$ for some $t\in B$, $s\in R$ and $w\in W_R$ such that $m_{s,t}$ is finite, we have $m_{s,t}=2$ and $\sigma=\tau$. Hence $m_{\sigma,\tau}$ is a diagonal entry of the Coxeter matrix of $\overline{W_B}$. Otherwise, by Theorem \ref{Gal}, we have $m_{\sigma,\tau}=m_{s,t}$ for some
$s, t\in B$ or $m_{\sigma,\tau}=\infty$. It follows that $\gamma(\overline{W}_B)=\gamma(W_B)$.
\end{proof}

\bigskip\noindent
{\it 2.9 Loops for Coxeter groups with a right-angled partition} 

\noindent Let $I$ be a subset of $S$. For any word or cyclic word $w=s_1s_2\dots s_n\in \mathcal{W}_S$, let $l_I(w)$ be the number of its letters in $I$. Therefore for any partition $S=R\sqcup B$ of $S$, we have

\centerline{$l(w)=l_R(w)+l_B(w)$.}

\begin{lemma}\label{shortpath} Let $S=R\sqcup B$ be a right-angled partition. Let $w$ be a cyclic loop on the alphabet $S$ such that $l_R(w)<\gamma(W_R)$ and $l_B(w)<\gamma(W_B)$. Then one of the following assertions holds

\begin{enumerate}
\item $w$ contains a substring $sus$, where $s\in B$ and $u$ is a reduced word on the alphabet $B(s)$, or

\item $w$ contains a substring $tt$, where $t\in R$.
\end{enumerate}
\end{lemma}

\begin{proof} If $l_B(w)=0$, then Assertion (2) holds by Corollary \ref{cyclicgirth}. 

From now on, let us assume that $l_B(w)>0$. As $w$ is a cyclic word, we  can write it as

\centerline{$w=u_1s_1u_2s_2\dots u_ks_k$,}

\noindent where $k=l_B(w)$, $s_1, s_2\dots s_k$ are in $B$ and $u_1,\dots, u_k$ are in $\mathcal{W}_R$.
(It is not excluded that some words $u_i$ have length zero.)  As usual, the indices $1,2\dots k$ are viewed as elements of $\Z/k\Z$. 

Set $v_1:=\overline{u_1}$, $v_2:=\overline{u_1u_2}$, $v_3:=\overline{u_1u_2u_3} \dots$. For any index $i\in \Z/k\Z$, set $\sigma_i=s_i^{v_i}$. Since $\mu_R(\overline{w})=1$, it follows that $v_n:=\overline{u_1u_2\dots u_n}=1$. Therefore the identity $\overline{w}=1$ is equivalent to

\centerline{$\sigma_1\sigma_2\dots\sigma_k=1$}

By definition, each $\sigma_i$ belongs to $\overline B$. Moreover, as a word on the alphabet $\overline B$, the word $\sigma_1\sigma_2\dots\sigma_k$ is a loop.

By Lemma \ref{bargirth} we have

\centerline{$\gamma(\overline{W}_B)=\gamma(W_B)$.}

\noindent
Since $k<\gamma(\overline{W}_B)$, it follows from Corollary \ref{cyclicgirth} that there are two indices $i,j\in\Z/k\Z$ such that $\sigma_i=\sigma_{i+1}$ and $\sigma_j=\sigma_{j+1}$. We can choose $i$ and $j$ in such a way that $l(u_{i+1})\leq l(u_{j+1})$. Since we have $l(u_{i+1})+l(u_{j+1})\leq l_R(w)<\gamma(W_R)$, we have

\centerline{$l(u_{i+1})<\gamma(W_R)/2$.}

\noindent Set $s=s_i$, $s'=s_{i+1}$ and $u=u_{i+1}$. The equality $\sigma_i=\sigma_{i+1}$ is equivalent to

\centerline{$s=s'^{\overline u}$.}

\noindent By Assertion (2) of Lemma \ref{parabolic}, we have $s=s'$. Therefore $w$ contains the substring

\centerline{$sus$,}

\noindent where $l(u)<\gamma(W_R)/2$.

To finish the proof, let us consider two cases.

\smallskip\noindent
{\it Case 1:} $u$ is not reduced. By Lemma \ref{girth}, the word $u$ contains the substring $tt$ for some $t\in I$, therefore Assertion (2) holds.

\smallskip\noindent
{\it Case 2:} $u$ is reduced. Since $s=s^{\overline u}$, it follows from the Assertion (1) of Lemma \ref{parabolic} that 
$u$ belongs to $W_R(s)$. Therefore $u$ is a word on the alphabet $R(s)$, and Assertion (1) holds.
\end{proof}

\bigskip\noindent
{\it 2.10 Homotopically trivial paths in $\Cay^+ W$}

\begin{thm}\label{null} Let $S=R\sqcup B$ be a right-angled partition. Let $w$ be a cyclic loop on the alphabet $S$ such that $l_R(w)<\gamma(W_R)$ and $l_B(w)<\gamma(W_B)$.

Then $\vert w\vert$ is null-homotopic in $\Cay^+\, W$.
\end{thm}

\begin{proof} 
If $w$ contains a substring $tt$, $\vert w\vert$ is homotopic in $\Cay\,W$ to the loop $w'$ obtained by deleting this substring. 

Otherwise, by Lemma \ref{shortpath}, $w$ contains a substring $sus$, where $s\in B$ and  $u$ is a reduced word on the alphabet $R(s)$. Hence by Lemma \ref{homotopic}, the loop $\vert w\vert$ is homotopic to the loop $w'$ obtained by replacing the substring $sus$ by $u$.

In both cases, we have $l(w')=l(w)-2$, while $l_R(w')<\gamma(W_R)$ and $l_B(w')<\gamma(W_B)$. Therefore by induction, $\vert w\vert$ is null-homotopic in $\Cay^+\,W$.
\end{proof}
\section{Uniformization of $2k$-Regular Tesselations}

\noindent Let $k$ be an integer. For simplicity, it will be assume that $k\geq 3$ in the whole section.
For $k=1$ or $2$, the theory is not difficult, but some details are slighty different.

We  define the notion of $2k$-regular standard hexagonal tesselation of a hyperbolic surface $\Sigma$. Except where stated otherwise, all tesselations considered in the paper have hexagonal tiles, so we will skip the term hexagonal in what follows.

We will first show a formal uniformization theorem for these tesselations.
There is a universal surface $\mathcal{H}$, endowed with such a tesselation, on which a certain Coxeter group $W(k)$ acts, such that any  
 closed surface with a $2k$-regular standard  tesselation is isometric to $\mathcal{H}/H$ for some finite index subgroup of $W(k)$. 
 Conversely, a finite index subgroup $H\subset W(k)$ for which the tesselation on $\mathcal{H}/H$ is $2k$-regular can be readily characterised.

 By definition, the {\it curves} of a tesselation are the maximal geodesics containing the sides of the tiles. All curves of a $2k$-regular tesselation have  length $2k\,\arcosh\,2$. This leads to the question - are the set of systoles and the set of  curves of the tesselation of $\mathcal{H}/H$ identical?    
This is partly answered by the main result of Section 3, namely the Criterion \ref{systole}.

The surface $\mathcal{H}$ has two realizations. The first one, denoted by $\mathcal{H}(k)$, is a quotient of the Poincar\'e 
half-plane $\HP$. The second realization is the Coxeter complex $\Cay^+\,W(k)$. 
The proof of Criterion \ref{systole} uses these two realizations of $\mathcal{H}$, and it combines
 standard hyperbolic trigonometry and Theorem \ref{null} of the previous section.

\bigskip
\noindent
{\it 3.1 The decorated hexagons $\mathcal{P}$ and $\overline{\mathcal{P}}$}

\noindent Let $\HP$ be the Poincar\'e half-plane, endowed with its hyperbolic metric. Recall that a hexagon $P\subset \HP$ is {\it regular} if its automorphism group is flag-transitive, i.e. if $\Aut\,P$ acts transitively on the pairs
$({\bf e},v)$, where ${\bf e}$ is a side and $v\in {\bf e}$ is a vertex of $P$. The following lemma follows readily from hyperbolic trigonometry, e.g.  \cite{ratcliffe} pp.90-96.

\begin{lemma}\label{hexagon} Up to isometry, there exists a unique hyperbolic right-angled hexagon $P$ whose side lengths are all equal. Moreover $P$ is regular and the common length of its sides is $L:=\arcosh\, 2$. 
\end{lemma}

\bigskip
By definition, the {\it decorated hexagon} $\mathcal{P}$ is the oriented hexagon $P$, whose  sides,
$S_1,S_2\dots S_6$ are indexed by $\Z/6\Z$ in an anti-clockwise direction. (The orientation of $P$ induces an orientation of its sides.) Let $\overline{\mathcal{P}}$ be the same hexagon with opposite orientation. By definition the {\it red sides} are $S_1, S_3$ and $S_5$, and the other three are called the {\it blue sides}.

\bigskip\noindent
{\it 3.2 The $2k$-regular standard  tesselations}

\noindent
Let $\Sigma$ be an oriented hyperbolic surface, finite or infinite, and let $\tau$  be a tesselation of $\Sigma$
whose tiles are the decorated hexagons $\mathcal{P}$ or $\overline{\mathcal{P}}$. The tesselation $\tau$ is called a {\it standard tesselation} if it satisfies
the following axioms

 \smallskip
(AX1) The tiles are glued along sides of the same index,

\smallskip
(AX2)  Each vertex of the tesselation has valence four.

\smallskip
In this definition, it should be understood that the tiles of a standard tesselation {\it are always  the decorated right-angled regular hexagons} $\mathcal{P}$ and $\overline{\mathcal{P}}$.
The second axiom implies that the sum of the four angles at each vertex is $2\pi$, so it  is equivalent to the fact that $\Sigma$ has no boundary. 

Let $\tau$ be a standard tesselation. By definition, a {\it curve of the tesselation} $\tau$ is a maximal
geodesic in the $1$-skeleton of $\tau$. Since the angle
between two adjacent edges of same index is $\pi$, each curve $C$
is a maximal geodesic of $\Sigma$ 
consisting of a union of adjacent edges of the same index $i$. This common index $i$ is called the  {\it index}
of the curve $C$. 

Given a positive integer $k$, a 
standard  tesselation $\tau$ is called $k$-{\it regular} if it satisfies the following additional requirement

\smallskip
(AX3) Each curve $C$ of the tesselation is a closed geodesic of length $kL$.

\smallskip\noindent
For a standard tesselation $\tau$ of $\Sigma$  satisfying (AX3), each curve $C$ of index $i$ alternately meets  
 curves of index $i+1$ and  curves of index 
$i-1$, so $k$ is an even integer. From now on, we will only speak about $2k$-regular tesselations to emphasise the fact that $2k$ is an even integer. 

Let $\mathcal{T}ess(\Sigma,2k)$ be the set of all $2k$-regular standard tesselations of $\Sigma$.
Once again, it is tacitly assumed that the tiles are the hexagons $\mathcal{P}$ and $\overline{\mathcal{P}}$.

\bigskip\noindent
{\it 3.3 The universal tesselated surface $\mathcal{H}(k)$}

\noindent  
Let $\HP$ be the Poincar\'e half-plane. By Lemma \ref{hexagon}, there is a unique right-angled regular hexagon. Hence, by the Poincar\'e polygon Theorem, there exists a unique (up to isometry) standard  tesselation $\tau_{\infty}$ of $\HP$. Let us choose one tile $\tilde T$ of the tesselation $\tau_{\infty}$ and let $\tilde{S}_i$  be its  side of index $i$. The tile $\tilde T$ will be called 
{\it the distinguished tile} of $\tau_{\infty}$. For $i\in\Z/6\Z$, let $\Delta_i$ be the line containing the side $\tilde{S}_i$
 and let $s_i$ be the reflection across the line $\Delta_i$. The subgroup of $\PGL_2(\R)$ generated 
by the six reflections $(s_i)_{i\in \Z/6\Z}$ is the right-angled Coxeter group $W(\infty)$ with presentation

$$\langle (s_i)_{i\in \Z/6\Z} \mid s_i^2=1 \,\mathrm{and}\, (s_is_{i+1})^2=1,
\,\forall i\in\Z/6\Z \rangle.$$

Let $k\geq 2$ and let $N(k)$ be the normal subgroup of $W(\infty)$ generated by the elements
$(s_is_{i+2})^k$, for $i\in\Z/6\Z$, and all their conjugates. The group $W(k):=W(\infty)/N(k)$ is the Coxeter group with presentation

$$\langle (s_i)_{i\in \Z/6\Z} \mid
s_i^2=1, \, (s_is_{i+1})^2=1,
\,\mathrm{and}\,(s_is_{i+2})^k=1,
\,\forall i\in\Z/6\Z 
\rangle$$

Set $W(\infty)^+:=W(\infty)\cap PSL_2(\R)=\{w\in W(\infty)\mid \epsilon(w)=1\}$.
It follows from the Poincar\'e Theorem that $W(\infty)$ acts freely and transitively on the set of tiles of $\tau_{\infty}$. So a nontrivial element $w\in W(\infty)^+$ is elliptic iff it is conjugate to $s_{i}s_{i+1}$ for some $i\in \Z/6\Z$. Since $k\geq 2$, the subgroup  $N(k)$ acts freely on $\HP$.

Set $\mathcal{H}(k)=\HP/N(k)$ and let $\tau_{k}$ be the standard tesselation of $\mathcal{H}(k)$ induced by $\tau_{\infty}$. Note that $s_{i-1}s_{i+1}(\Delta_i)=\Delta_i$ and its restriction to  the line $\Delta_i$ is a translation of length $2L$. It follows that $\tau_{k}$ is a $2k$-regular tesselation. The {\it distinguished tile} of $\mathcal{H}(k)$, denoted by $T$, is the image of $\tilde T$ in $\mathcal{H}(k)$.

Set $W(k)^+=\{w\in W(k)\mid \epsilon(w)=1\}$ and set $t_i=s_{i-1}s_{i+1}$ for any $i\in\Z/6\Z$. For any element $t\in W(k)$, let $t^{W(k)}$ be its conjugacy class. When a subgroup $H\subset W(k)^+$ acts freely on $\mathcal{H}(k)$,
let $\tau_{H}$ be the tesselation induced by $\tau_{k}$ on the surface $\mathcal{H}(k)/H$. Note that the condition $H\subset W(k)^+$ ensures that $\mathcal{H}(k)/H$ is orientable.

\begin{lemma}\label{condition}
A subgroup $H\subset W(k)^+$ acts freely on
$\mathcal{H}(k)$ iff 

(\ref{condition}.1) \hskip2cm{$H\cap (s_is_{i+1})^{W(k)}=\emptyset$ for all $i\in\Z/6\Z$.}

Moreover assume that $H$ satisfies the condition (\ref{condition}.1). Then the tesselation $\tau_{H}$ is $2k$-regular iff the following condition holds

(\ref{condition}.2) \hskip16mm{$H\cap(t_i^l)^{W(k)}=\emptyset$ for all $i\in\Z/6\Z$ and $1\leq l<k$.}
\end{lemma}

\begin{proof} Let $w\in H$. Since $W(k)$ is tile-transitive, $w$ has a fixed point in $\mathcal{H}(k)$ iff $w^v$ has a fixed point in $T$, for some $v\in W(k)$. By hypothesis, $w^v$ is not a reflection therefore $w^v=s_is_{i+1}$ for some $i$ in $\Z/6\Z$, which proves the first assertion.

Assume that $\tau_{H}$ is not $2k$-regular. 
By assumption, there is a curve $C$ of the tesselation $\tau_k$ whose image in $\mathcal{H}(k)/H$ has length less than $kL$.
Since $W(k)$ is tile-transitive, we can assume that
$C$ contains  the side $S_i$ of the distinguished tile $T$, for some $i\in\Z/6\Z$. We have 
$h(C)=C$, for some nontrivial $h\in H$. Since
it has no fixed points, 
$h\vert_{C}$ is a rotation. It follows that
 $h=t_i^l$ for some $\in\Z/6\Z$, which proves the second assertion.
\end{proof}

Set

\centerline{$\mathcal{T}ess(*,2k)=\cup_{\Sigma}\, \mathcal{T}ess(\Sigma,2k)$,}

\noindent where $\Sigma$ varies over all oriented closed hyperbolic surfaces of genus $g\geq 2$.

In what follows, we will only use the previous Lemma \ref{condition}. However we would like to briefly explaine that $\mathcal{H}(k)$ is the universal cover of all $2k$-regular standard tesselations, as will now be shown.

\begin{thm}\label{universal} 
 The map $H\mapsto \tau_{H}$ is a one-to-one correspondence between 
 \begin{enumerate}
\item the finite index subgroups $H\subset W(k)^+$ 
satisfying (\ref{condition}.1) and
(\ref{condition}.2), and 

\item the $2k$-regular standard tesselations $\tau$ of closed oriented surfaces. 
\end{enumerate}
\end{thm}

In the previous statement,  it should be understood that the word ``subgroups" refers to conjugacy classes of subgroups and ``tesselations" refers to isometry classes of tesselations. 

\begin{proof} The Lemma \ref{condition} shows that this map is well-defined. 

We will now define the inverse map. Let $\tau\in \mathcal{T}ess(*,2k)$. By definition, $\tau$ is a  tesselation of some oriented closed surface $\Sigma$.  The tesselation $\tau$ induces a standard tesselation of its universal cover $\HP$.
Since the induced tesselation is isometric to $\tau_\infty$,  there is an  embedding 
$\pi_1(\Sigma, p)\subset W(\infty)$, where
$p\in\Sigma$ is a base point. Since $\tau$ is $2k$-regular, it follows that $\pi_1(\Sigma, p)\supset N(k)$, and therefore
$\tau=\tau_{_H}$, where $H=\pi_1(\Sigma, p)/N(k)$.
\end{proof}

\noindent {it Remark.} The universal tesselated surfaces $\mathcal{H}(k)$
can be defined for $k=1$ and $2$. In fact
$\mathcal{H}(1)$ is the genus $2$ surface of 
\cite{SS}, which was the starting point of our paper.

\bigskip\noindent
{\it 3.4 The homeomorphism $Cay^+ W(k)\simeq \mathcal{H}(k)$}

\noindent 
Since $W(k)$ acts freely and transitively on the set of tiles, the Cayley graph $\Cay\, W(k)$ can be identified with the dual graph of the tesselation $\tau_{k}$. 

Let us recall more precisely the definition of the {\it dual graph} $\tau_{k}^*$ of $\tau_{k}$. Let $T$ be the distinguished tile of $\mathcal{H}(k)$ and $X$ be its center. For $w\in W(k)$, set $T(w)=w.T$, $X(w)=w.X$
and, for any $i\in\Z/6\Z$, let $S_i(w)$ be the side of $T(w)$ of index $i$.

By definition,

\centerline 
{$V(\tau_{k}^*):=\{X(w)\mid w\in W(k)\}$}

\noindent is the {\it set of vertices} of $\tau_{k}^*$. 
For $w\in W(k)$ and $i\in\Z/6\Z$, 

\centerline{${\bf O}:=T(w)\cup T(ws_i)$}

\noindent is a right-angled octogon with two sides of length $2L$ and all others of length $L$. Let ${\bf e}(w,ws_i)$ be the geodesic arc joining $X(w)$ and $X(ws_i)$ in ${\bf O}$. In this instance, the arcs are not oriented, so ${\bf e}(w,ws_i)={\bf e}(ws_i,w)$. By definition

\centerline{$E(\tau^*_{k}):=\{{\bf e}(w,ws_i)\mid w\in W(k)\,
\mathrm{and} \,i\in\Z/6\Z\}$}

\noindent is the {\it set of edges} of 
$\tau_{k}^*$.
The duality property means that each tile 
$T(w)$ contains exactly one vertex of $\tau_{k}^*$, namely $X(w)$, and each of its sides $S_i(w)$ meets
exactly one edge of $\tau_{k}^*$, namely 
${\bf e}(w,ws_i)$. The natural homeomorphisms between
an edge $\vert (w,ws_i)\vert$
of $\Cay\, W(k)$ and an edge 
${\bf e}(w,ws_i)$ of $\tau_{k}^*$ provide a natural homeomorphism

\centerline{$\Cay\, W(k)\simeq \tau_{k}^*$.}

It follows that the topological graph $\Cay\, W(k)$ is embedded in $\Cay^+ W(k)$ and in 
$\mathcal{H}(k)$.
In fact, the  
two embeddings of $\Cay\, W(k)$ are the same, as it is shown in the next

\begin{lemma}\label{homeo}
The embedding $\Cay\,W(k)\subset \mathcal{H}(k)$ extends to a  homeomorphism

\centerline{$\Cay^+ W(k)\simeq \mathcal{H}(k)$.}
\end{lemma}

\begin{proof} By definition, $\Cay^+ W(k)$ is tesselated by  quadrilaterals and  
$\mathcal{H}(k)$ is tesselated by hexagons.
Roughly speaking, it will be shown that these two tesselations are dual to each other, see Figure 
\ref{Figure2SystoleStory}.

Since $k\geq 3$, the rank-two commutative parabolic subgroups of $W(k)$ are the subgroups
$W_{I(i)}$, where
$I(i)=\{s_i, s_{i+1}\}$ and $i\in \Z/6\Z$. 
For $v\in W(k)$, it is clear that

\centerline{
${\bf H}=T(v)\cup T(vs_i)
\cup T(vs_{i+1})\cup T(vs_i s_{i+1})
=W(k)_{I(i)}.T(v)$}

\begin{figure}
\centering
\includegraphics[width=0.4\textwidth]{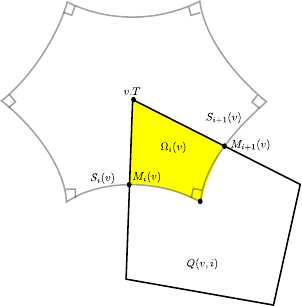}
\caption{The quadrilateral from the proof of Lemma \ref{homeo}.}
\label{Figure2SystoleStory}
\end{figure}

\noindent is a right-angled $12$-gon with four sides of length $2L$ and all others of length $L$. 
Let ${\bf Q}(v, i)$ be the quadrilateral contained in ${\bf H}$ whose set of vertices is
$$W(k)_{I(i)}.X(v)= \{X(v),X(vs_i), X(vs_{i+1}), X(vs_i s_{i+1})\}$$ as shown in Figure \ref{Figure2SystoleStory}. Consequently, we have ${\bf Q}(v, i)={\bf Q}(v', i)$ if $v= v'\mod W(k)_{I(i)}$.

Set $A_i(v)=S_i(v)\cap S_{i+1}(v)$ and let $M_i(v)$ be the midpoint $S_i(v)$ and $\Omega_i(v)\subset T(v)$ be the convex
 quadrilateral with vertices given by $M_i(v)$, $A_i(v)$, $M_{i+1}(v)$ and $X(v)$.
The tile $T(v)$ is tesselated by the six quadrilaterals  $\Omega_i(v)$. Since
 
 \centerline{${\bf Q}(v, i)\cap T(v)=\Omega_i(v)$}
 
\noindent the set of quadrilaterals

\centerline{$\{{\bf Q}(v, i)\mid  i\in\Z/6\Z
\,\mathrm{and}\,v\in W(k)\}$}

\noindent
tesselates $\mathcal{H}(k)$. 

Since $\Cay\, W(k)$ has been identified with the graph $\tau^*_{k}\subset \mathcal{H}(k)$, there is an equality

\centerline{ $\partial {\bf Q}(v, i)=
\partial {\bf c}(v, I_i)$.}

\noindent 

\noindent which can be extended to a homeomorphism 

\centerline{${\bf Q}(v, i)\simeq 
 {\bf c}(v, I_i)$.}
Therefore $\Cay^+ W(k)$ is homeomorphic to
 
\centerline {$\cup_{i,v} {\bf Q}(v, i)$}

\noindent where $i$ varies over $\Z/6\Z$ and $v$ over $W(k)$. It follows that $\Cay^+ W(k)$ is homeomorphic to $\mathcal{H}(k)$.
\end{proof}

\bigskip\noindent
{\it 3.5 The combinatorial datum $\omega(\gamma)$ associated to  arcs in $\mathcal{H}(k)$}

\noindent We will now start to investigate the
length of the closed geodesics of the surfaces 
$\mathcal{H}(k)/H$. To do so, we will first
look at the lengths of the arcs $\gamma$ in 
$\mathcal{H}(k)$. In this section, we will associate  a word $\omega(\gamma)$ over the letters $(s_i)_{i\in\Z/6\Z}$. Its length will be called
the combinatorial length of an arc.
Then, we will find a lower bound
for the combinatorial length of the closed geodesic. The relation between the hyperbolic lengths and the  hyperbolic lengths will be examined in the next subsection.
\\

We will now provide the precise definition
of the combinatorial datum $\omega(\gamma)$ associated to some geodesic paths $\gamma$.
  Let 
$\Arc_T(\mathcal{H}(k))$ be the set of geodesic arcs $\gamma:[0,l]\to \mathcal{H}(k)$ of length $l$ such that 

\begin{enumerate}
\item $\gamma(t)$ lies in $T^0$ for 
small $t\neq 0$

\item $\gamma(0)\in C$  and $\gamma(l)\in C'$ for some curves $C, C'$ of $\tau_k$,
\end{enumerate}

\smallskip
\noindent where $T^0$ is the interior of the distinguished cell $T$ of $\tau_{k}$.
Note that the first condition implies that $\gamma$ is not an arc of a curve of 
$\tau_{k}$. Of course, the previous conditions do not ensure that $\gamma$ necessarily lifts a closed geodesic
in some $\mathcal{H}(k)/H$. It is the case
only if
the indices of $C$ and $C'$ are equal together with some position and angle conditions
for $\gamma(0)$ and $\gamma(l)$. In the notation
$\Arc_T(\mathcal{H}(k))$ the index $T$ emphasizes 
that $\gamma$ originates in $T$. 
\\

Now we are going to define a word

\centerline{
$\omega(\gamma)=s_{i_1} s_{i_2}\dots s_{i_N}$,}

\noindent associated with the path $\gamma$.
First assume that $\gamma$ does not meet any vertex
of the tesselation. Then $i_1\dots, i_{N-1}$ are the indices of the curves successively crossed by 
$\gamma$ and $i_N$ is the index of the the curve $C'$ passing through $\gamma(l)$. Note that 
$\omega(\gamma)$ does not encode the index of the initial curve $C$.

When $\gamma$ does meet some vertices $v$,
we need a convention to precise the order of the curves that  $\gamma$ meets. When $\gamma$ crosses a vertex   at the intersection of two curves  of indices $i$ and $i+1$, we consider that $\gamma$ first crosses the curve of index $i$ and then the curve of index $i+1$. Similarly, if $\gamma(l)$ terminates at such point,  we consider that $\gamma$ first crosses the curve of index $i$ and then terminates on a curve of index $i+1$. 
As before, the datum $\omega(\gamma)$ does not encode any information about  $\gamma(0)$.

  The integer $N=l(\omega(\gamma))$ is called the {\it combinatorial length} of $\gamma$. 

\begin{lemma}\label{homotopy}
Let $\gamma\in \Arc_T(\mathcal{H}(k))$ be a closed geodesic. Then $\gamma$ is freely homotopic to the loop $\vert \omega(\gamma)\vert$ in $\Cay\,W(k)$.
\end{lemma}

\begin{proof} Set 
$\omega(\gamma)=s_{i_1} s_{i_2}\dots s_{i_N}$,
where $N$ is the combinatorial length of $\gamma$.
Let $0<t_1\leq\dots\leq t_n\dots \leq t_{N-1}\leq l$ be the 
successive time at which $\gamma(t_n)$ crosses
a curve of index $i_n$. Also set
$t_0=0$ and $t_N=l$.
For any $n$ with $1\leq n\leq N$, let $\gamma_n$ be the restriction of $\gamma$ to $[t_{n-1}, t_n]$, so we have

\centerline{$\gamma=\gamma_1*\gamma_2*\dots*\gamma_N$,}

\noindent where the $*$ denotes the concatenation
of paths.

Recall that, for $v\in W(k)$, $T(v)$ denotes the tile $v.T$ and $X(v)$ is its center.
Let $v_0, v_1\dots v_N$ be the elements of $W(k)$ defined by
$v_0=1$ and $v_n=\overline{s_{i_1} s_{i_2}\dots s_{i_n}}$ for $n\geq 1$. 
By definition,  
$\gamma(t_n)$ belongs to the side
$T(v_n)\cap T(v_{n-1})$ for any $0\leq n\leq N$. Let  $\delta_n\subset T(v_n)$ be the oriented geodesic arc from $\gamma(t_n)$ to $X(v_n)$, and let $\overline {\delta}_n$ be the same arc with the opposite orientation.

Since $\gamma$ is a loop, we have $v_N=1$ and therefore ${\delta}_N= \overline {\delta}_0$. It follows that $\gamma$ is freely homotopic to

\centerline{$\overline {\delta}_0*\gamma_1*\delta_1
*\overline {\delta}_1*\gamma_1*\dots
*\gamma_N*\delta_N
=\tilde{\gamma}_1*\tilde{\gamma}_2*\dots*\tilde{\gamma}_N$,}

\begin{figure}
\centering
\includegraphics[width=0.6\textwidth]{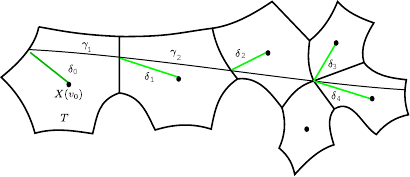}
\caption{The concatenation of arcs from the proof of Lemma \ref{homotopy}.}
\label{Figure3}
\end{figure}

\noindent where 
$\tilde{\gamma}_n=\overline {\delta}_{n-1}*\gamma_n*\delta_n$ for any $n\geq 1$. This is illustrated in Figure \ref{Figure3SystoleStory}. By definition $\tilde{\gamma}_n$ is a path joining  $X(v_{n-1})$ to $X(v_n)$. Since $\tilde{\gamma}_n$ lies in the convex octogon $T(v_{n-1}\cup T(v_n)$,  it is homotopic to the geodesic arc ${\bf e}(v_{n-1}, v_n)$ going from $X(v_{n-1})$ to $X(v_n)$.

Hence $\gamma$ is freely homotopic to

\centerline{
${\bf e}(v_{0}, v_1)*{\bf e}(v_1, v_2)
\dots *{\bf e}(v_{N-1}, v_N)$,}

\noindent which is precisely the loop 
$\vert \omega(\gamma)\vert$ in $\Cay\,W(k)$.
\end{proof}

Set $S=\{s_i\mid i\in\Z/6\Z\}$,
$R=\{s_1,s_3,s_5\}$ and $B=\{s_2,s_4,s_6\}$.
The elements of $R$, respectively of $B$, are called the {\it red letters}, respectively the blue letters.
For any word  
$w=s_{i_1}s_{i_2}\dots s_{i_N}\in\mathcal{W}_S$,
let $l_R(w)$ (respectively $l_R(w)$)  be the number of red (respectively blue) letters in $w$.

\begin{cor}\label{RBlength} Assume that $k$ is even.
Let 
$\gamma\in\Arc_T(\mathcal{H}(k))$ be a closed geodesic. Then we have

\centerline{ 
$l_R(\omega(\gamma))\geq 2k$ or  $l_B(\omega(\gamma))\geq 2k$.}
\end{cor}

\begin{proof} Assume otherwise, i.e. $l_R(\omega(\gamma))< 2k$ and $l_B(\omega(\gamma))<2k$. Since $k$ is even, the partition $S=R\cup B$ is right-angled. Therefore by Theorem \ref{null}, the path $\vert \omega(\gamma)\vert$ is null-homotopic in $\Cay^+\,W(k)$. By Lemma \ref{homeo} $\Cay^+\,W(k)$ is homeomorphic to 
$\mathcal{H}(k)$ and by Lemma \ref{homotopy} $\gamma$ is freely homotopic to $\vert \omega(\gamma)\vert$. This contradicts the fact that no closed geodesic is null-homotopic on a hyperbolic surface.
\end{proof}

\bigskip\noindent
{\it 3.6 Combinatorial length versus geometric length}

\noindent
Let $\gamma\in\Arc_T(\mathcal{H}(k))$ be a geodesic arc
of $\mathcal{H}(k)$. In this section, we will compare the metric length
$L(\gamma)$ of $\gamma$ with its combinatorial length $l(\omega(\gamma))$.

\begin{lemma}\label{geolength1} Let $\gamma\in\Arc_T(\mathcal{H}(k))$ be a geodesics arc. We have $L(\gamma)>L$ whenever one of the following conditions is satisfied
\begin{enumerate}
\item $l(\omega(\gamma))=1$ and $\gamma$ joins two non-consecutive sides of a tile, or

\item $l(\omega(\gamma))=2$.
\end{enumerate}
\end{lemma}

\begin{proof} The first statement is well known, see e.g. \cite{SS}\cite{FB}. Let $\gamma\in\Arc_T(\mathcal{H}(k))$ be an arc  and let
$\gamma=\gamma_1*\gamma_2$ be its factorization into arcs of combinatorial length 1. 
If $\gamma_1$ or $\gamma_2$ join two non-consecutive sides of a tile, then it is already proved that $L(\gamma)>L$.

Otherwise, after applying some isometry, we can assume that $\gamma(0)$ belongs to
$S_1(1)$, then $\gamma$ crosses $S_2(1)$ and ends  on $S_3(s_2)$, where $S_i(v)$ denotes the side of index $i$ of the tile $T(v)$. Since $T(1)\cup T(s_2)$ is a convex octogon, we can lift it to the Poincar\'e
half plane. Let $\Delta_1$ and $\Delta_3$ be the line containing the arcs $S_1(1)$ and $S_3(s_2)$.
The lift of $\gamma$ joins  $\Delta_1$ and $\Delta_3$. Since
$S_2(1)=S_2(s_2)$ is the common perpendicular to $\Delta_1$ and $\Delta_3$, we have $L(\gamma)>L(S_2(1))=L$.
\end{proof}

\begin{lemma}\label{geolength2} Let 
$\gamma\in\Arc_T(\mathcal{H}(k))$ be a closed geodesic. If $l_R(\omega(\gamma))\geq 2k$ or 
$l_B(\omega(\gamma))\geq 2k$, then $\gamma$ has length $>2kL$. 
\end{lemma}

\begin{proof}
We can assume that $l_R(\omega(\gamma))\geq 2k$, and write $\gamma=\gamma_1*\gamma_2*\dots*\gamma_N$, where $N=l_R(\omega(\gamma))$ and each $\gamma_n$ is an arc joining a red side to another red side. 

If $l(\gamma_n)=1$, then $\gamma$ joins two red sides of the same tile. Since these sides are not consecutive, we have $L(\gamma_n)>L$ by  the part (1) of Lemma \ref{geolength1}. Otherwise, we have $l(\gamma_n)\geq 2$ and $L(\gamma_n)>L$ by the part (2)
of Lemma \ref{geolength1}. Therefore each arc $\gamma_n$ has length $>L$. It follows that 

\centerline{$L(\gamma)=
\sum_{1\leq n\leq N}\,L(\gamma_n)>NL\geq 2kL$.} 
 \end{proof}

\bigskip\noindent
{\it 3.7 Systoles of $\mathcal{H}(k)/H$}

\noindent
Assume now that $k$ is even. It follows from 
Corollary \ref{RBlength} and
Lemma \ref{geolength2} that the systoles of
${\mathcal H}(k)$ are the curves of the tesselation. Let $H\subset W(k)^+$ be a subgroup satisfying the conditions (11.1) and 
(11.2). Then all curves of the tessalation
 in ${\mathcal H}(k)/H$
have length $2kL$.  Determining when these curves of $\tau_{H}$ are the systoles is a delicate
question. The next criterion  provides a partial answer.

For $w\in W(k)$, set $H^w=wHw^{-1}$. Also set

\centerline{$B_{4k}:=\{w\in W(k)\mid l(w)\leq 4k\}$.}

\begin{crit}\label{systole} Let $H\subset W(k)^+$ be a subgroup such that $H^w\cap B_{4k}=\{1\}$ for all $w\in W(k)$.

Then $H$ acts freely on $\mathcal{H}(k)$ and the tesselation $\tau_{H}$ is $2k$-regular. Moreover, the set of systoles of $\mathcal{H}(k)/H$ is exactly the set of curves of $\tau_{H}$.
\end{crit}

\begin{proof} By Lemma \ref{condition}, $H$ acts freely on $\mathcal{H}(k)$ and the tesselation $\tau_{H}$ is $2k$-regular. The curves of the tesselation  $\tau_{H}$ have length $2kL$. Hence, it remains to prove that  any closed geodesic $\overline{\gamma}$ of $\mathcal{H}(k)/H$, which is not a curve, has length $L(\overline{\gamma})>2kL$.

The choice of a distinguished tile $\overline {T}$ was arbitrary,
so we can assume that $\overline{\gamma}$ 
intersects the interior of $\overline {T}$.
 Hence there is a geodesic arc $\gamma\in\Arc_T(\mathcal{H}(k))$ which lifts $\overline{\gamma}$. Since $L(\overline{\gamma})=L(\gamma)$, it is enough to show that $L(\gamma)>2kL$.

If $l(\omega(\gamma))>4k$, then by Lemma \ref{geolength1}, we have $L(\gamma)>2kL$. Assume now that $l(\omega(\gamma))\leq 4k$. Since we have 

\centerline{$H^w\cap B_{4k}=\{1\}$ 
for all $w\in W(k)$,}

\noindent it follows that $\gamma$ is a closed geodesic of $\mathcal{H}(k)$. By Corollary \ref{RBlength}, we have 

\centerline{$l_R(\omega(\gamma))\geq 2k$ or 
$l_B(\omega(\gamma))\geq 2k$,}

\noindent hence, by Lemma \ref{geolength1}, we have 
$L(\gamma)>2kL$.
\end{proof}
\section{The subgroup $H(k)$ of $W(k)$}

\noindent For simplicity, we will assume that the integer $k$ is $\geq 3$. Set  $K:=\Q(\cos \pi/k)$ and let $\mathcal{O}$ be the ring of integers of the field $K$. 

In this section, we use the Tits representation $\rho:W(k)\to GL_6(K)$ to find a subgroup $H(k)$ of $W(k)$ satisfying the hypothesis of Criterion \ref{systole} with index $[W(k):H(k)]\leq 3^{72 k\phi(2k)}$, see Proposition \ref{subgroup}.

\bigskip\noindent
{\it 4.1 The Tits representation}

\noindent In \cite{T61}, Tits defined a faithful representation  of any Coxeter group.
References  are \cite{T61} or \cite{Da} Appendix D. Here we will 
describe his result for the groups $W(k)$.

Let $(\alpha_i)_{i\in \Z/6\Z}$ be a basis of the six-dimensional vector space $K^6$. There is a symmetric bilinear form $B$ on $K^6$ given by
\begin{enumerate}
\item $B(\alpha_i,\alpha_i)=2$

\item $B(\alpha_i,\alpha_j)=0$, if $i-j=\pm1$,

\item $B(\alpha_i,\alpha_j)=-2\cos(\pi/k)$,
$i-j=\pm2$, and 

\item $B(\alpha_i,\alpha_j)=-2$, if 
$i-j=\pm 3$.
\end{enumerate}

\noindent For any $\alpha\in K^6$ with $B(\alpha,\alpha)=2$, let $s_\alpha$ be the  hyperplane reflection, defined by

\centerline{
$s_\alpha(\lambda)=
\lambda-B(\alpha,\lambda) \alpha$,}

\noindent for any $\lambda\in K^6$. The {\it Tits representation} of $W(k)$ is the group homomorphism 

\centerline{$\rho: W(k)\to \GL_6(K)$}

\noindent  defined on the generators by 
$\rho(s_i)=s_{\alpha_i}$, for any $i\in \Z/6\Z$. 

\begin{thm}\label{Tits61} (Tits \cite{T61}\cite{Da}) 
The representation  $\rho$ is faithful.
\end{thm}

\bigskip\noindent
{\it 4.2 The $l_\infty$-norms of $K$ and $\End(K^6)$}

\noindent Let $\mathcal{F}$ be the set of field embeddings $v:K\to \R, x\mapsto x_v$. The field $K$ is totally real, and its degree is $\phi(2k)/2$, where $\phi$ is the Euler totient function. It follows that $\Card\mathcal{F}=\phi(2k)/2$.

For $x\in K$, let $\lVert x\rVert$ be its {\it $l_\infty$-norm} defined by

\centerline{
$\lVert x\rVert=\Max_{v\in\mathcal{F}} \,\vert x_v\vert$.}

\noindent  We have 
$\lVert xy\rVert\leq \lVert x\rVert \lVert y\rVert$.
This norm should not be confused with the usual norm 
$N_{K/\Q}(x):=\prod_{v\in\mathcal{F}}\, x_v$, which is a determinant. Let $x\in \mathcal{O}\setminus \{0\}$. Since
 $\vert N_{K/\Q}(x)\vert $ is a positive integer,   we have  

\centerline{$\lVert x\rVert\geq 1$ if $x\in \mathcal{O}\setminus \{0\}$.} 

\noindent For $i,\,j\in\Z/6\Z$, let 
$E_{i,j}\in \End(K^6)$ be the linear map

\centerline{
$E_{i,j}:v\in K^6\mapsto 
\langle\alpha_j\mid v\rangle\, \alpha_i$.}

Since $B(\alpha_i,\alpha_j)$ is a circulant matrix, an easy computation shows that

\centerline{$\det\,B(\alpha_i,\alpha_j)=
64(3\cos^4(\pi/k)-4\cos^3(\pi/k)$}
\noindent where $\mu_{6}$ is the group of $6^{th}$ roots of unity.
 Since $k\geq 3$, the bilinear form $B$ is 
 nondegenerate and  the set 

\centerline{$\{E_{i,j}\mid i,j\in\Z/6\Z\}$}

\noindent
is a basis of $\End(K^6)$. Therefore any element 
$A\in \End(K^6)$ can be written as 
$A=\sum_{i,j\in\Z/6\Z}\,a_{i,j}\,E_{i,j}$,
where $a_{i,j}\in K$.
Its {\it $l_\infty$-norm} $\lVert A\rVert$ is defined by

\centerline{
$\lVert A\rVert=\Max_{i,j\in\Z/6\Z}
 \,\lVert a_{i,j}\rVert$.}
 
\noindent 
For each $i\in\Z/6\Z$, set $F_i=E_{i,i}$ and 
for any word $w=i_1\dots i_n$ on the alphabet
$\Z/6\Z$, set

\centerline{$F_w=F_{i_1}\dots F_{i_n}$.}

\noindent  The $l_\infty$-norms of  $\End(K^6)$
is not multiplicative, i.e.
$\lVert AB\rVert$ is not necessarily $\leq \lVert A\rVert \lVert B\rVert$. Nevertheless, 
$\lVert F_w\rVert$ can still be estimated, as shown in the next lemma.

\begin{lemma}\label{operatornorm} Let $w$ be a word of length $n$ over
the alphabet $\Z/6\Z$. We have

\centerline{$\lVert F_w\rVert\leq
 2^n$.}
\end{lemma}

\begin{proof} The Galois conjugates of $\cos(\pi/k)$ are the numbers $\cos(m\pi/k)$, for $m$ prime to $k$, hence we have 
$\lVert B(\alpha_i,\alpha_j)\rVert\leq 2$ for any
$i,j\in\Z/6\Z$. Since
 $E_{i_1,j_1}E_{i_2,j_2}=
 B( \alpha_{j_1}, \alpha_{i_2}) E_{i_1,j_2}$,
 for any $i_1,j_1,i_2,j_2\in \Z/6\Z$,  
 it can be proven by induction over $n$ that
 $F_w= a_w E_{i_1,i_n}$ for some $a_w\in \mathcal{O}$ with
 $\lVert a_w\rVert\leq
 2^{n-1}<2^n$, from which the lemma follows.
 \end{proof}

\bigskip\noindent
{\it 4.3 The  $l_\infty$-norm of the Tits representation}

\noindent Recall that $\rho:W(k)\to\GL_6(K)$ denotes the Tits representation.

\begin{lemma} For any $w\in W(k)$, we have

\centerline{$\lVert \rho(w)-1\rVert<
\, 3^{l(w)}$.}
\end{lemma}

\begin{proof} Set $n=l(w)$ and let $w=s_{i_1}\dots s_{i_n}$ be a reduced decomposition of $w$.
Let $\mathcal{V}$ be the collection of all nonempty subwords of the word $s_{i_1}\dots s_{i_n}$. For $l>0$, set $\mathcal{V}_l=\{v\in \mathcal{V}\mid l(v)=l\}$.
Since some subwords appear more than once, $\mathcal{V}$ and $\mathcal{V}_l$ are sets with multiplicity. For example, if
$w=s_1s_2s_3s_2$, the subword $s_1s_2$ appears twice. 

We have $\rho(s_i)=1-F_i$, for any $i\in\Z/6\Z$. Hence we obtain

\centerline{$\rho(w)-1=\sum_{v\in \mathcal{V}}
\,(-1)^{l(v)} F_v$, and}

\centerline{$\lVert\rho(w)-1\rVert\leq \sum_{v\in \mathcal{V}}
\lVert F_v\rVert$.}

\noindent 
By Lemma \ref{operatornorm}, we have

\centerline{
$\lVert F_v\rVert\leq
 2^{l(v)}$.}

\noindent Since $\Card \mathcal{V}_l=(^n_l)$, we obtain

\centerline{
$\lVert \rho(w)-1\rVert\leq\sum_{1\leq l\leq n} (^n_l)2^l=(3^n-1) <3^n$.}
\end{proof}

\bigskip\noindent
{\it 4.4 Finite quotients of $W(k)$}

\noindent Set $\mathcal{R}:=
\oplus_{i,j\in\Z/6\Z}\,\mathcal{O} E_{i,j}$.
Since $E_{i_1,j_1}E_{i_2,j_2}=
 B( \alpha_{j_1}, \alpha_{i_2}) E_{i_1,j_2}$,
 $\mathcal{R}$ is a subring of $\End(K^6)$.
 Since $s_{\alpha_i}=1-E_{i,i}$, we have
 $\rho(W(k))\subset \mathcal{R}^*$, where
 $\mathcal{R}^*$ is the group of invertible elements of
 $\mathcal{R}$. Set

\centerline{$H(k)=\{w\in W(k)\mid \rho(w)-1\in
3^k\mathcal{R}\}$.}

For any integer $n>0$, set

\centerline{$B_n=\{w\in W(k)\mid l(w)\leq n\}$.}

\begin{prop}\label{subgroup} The subgroup $H(k)$ is normal and lies inside $W(k)^+$. Moreover we have

\begin{enumerate}
\item $B_{4k}\cap H(k)=\{1\}$, and

\item $\Card W(k)/H(k)\leq 3^{72 k\phi(2k)}$.
\end{enumerate}
\end{prop}

\begin{proof}

\noindent
The group 
$\mathcal{R}(k):=\mathcal{R}^*\cap(1+3^k\mathcal{R})$
is a normal subgroup of $\mathcal{R}^*$, hence
$H(k)$ is normal.  Since  we have $\det \rho(h)\equiv 1\mod 3$, for any $h\in H(k)$, the group $H(k)$ lies in $W(k)^+$. It remains to prove Assertions (1) and (2).

\smallskip\noindent
{\it 1. Proof that $B_{4k}\cap H(k)=\{1\}$.}
Let $w\in B_{4k}\cap H(k)$. By definition, we have
$\rho(w)-1=3^kA$, where $A$ belongs to $\mathcal{R}$.
.  
By Lemma \ref{operatornorm}, we have

$$\lVert 3^k A\rVert=\lVert\rho(w)-1\rVert<
3^{4k}=N,$$

\noindent therefore $\lVert A\rVert<1$.
Since for any $a\in\mathcal{O}$, we have $a=0$ or 
$\lVert a\rVert\geq 1$ and since
$\lVert A\rVert=\Max \lVert a_{i,j}\rVert$, we have $A=0$.  By Theorem \ref{Tits61}, it follows that $w=1$.

\smallskip\noindent
{\it 2. Proof that  $\Card W(k)/H(k)\leq 3^{72 k\phi(k)}$.}

By definition, there is an embedding

\centerline{$W(k)/H(k)
\subset \mathcal{R}/N\mathcal{R}$.}

\noindent As an abelian group, $\mathcal{O}$ is isomorphic to
$\Z^{\phi(k)/2}$, hence 
$\mathcal{R}/N\mathcal{R}\simeq \Z/N\Z^{18 \phi(k)}$. It follows that

\centerline{
$\Card W(k)/H(k)\leq N^{18 \phi(k)}=3^{72 k\phi(k)}$.}
\end{proof}

\section{Bounds on $\Fill(g)$}

The last step of the proof of the bound on $\Fill(g)$ stated in the Introductin involves the factor
$57/\sqrt{\ln\ln\ln g}$. This will now be shown to be a consequence of a 1904 result of E. Landau.

\bigskip\noindent
{\it 5.1 The set $B$ and Landau's Theorem}

\noindent
Let $p_1<p_2<\dots$ be the ordered list of all odd prime numbers. For any $n\geq 1$, set
$q_n=2\prod_{k=1}^n\,p_i$, and set
$B:=\{q_1, q_2\dots\}=\{6,30,210\dots\}$. The following classical theorem is an improvement  of the prime number Theorem of de la Vall\'ee Poussin \cite{VP} and Hadamard \cite{H}.

\begin{thm} (Landau \cite{L}\label{Landau}) While $k$ varies over $B$, we have

$$\phi(k)\sim  e^{-\gamma}{k\over\ln\ln k},$$

\noindent where $\gamma=0.577\dots$ is Euler's constant.
\end{thm}

For the proof, see \cite{HW}, Theorem 328 p. 352.

\bigskip\noindent
{\it 5.2 The constant $\delta=12\sqrt{e^{-\gamma}\ln 3}=9.42\dots$} 

\noindent 
Let $k\geq 3$. By Proposition \ref{subgroup}, 
there is a subgroup $H(k)\subset W(k)^+$ such that

\begin{enumerate}
\item $B_{4k}\cap H(k)=\{1\}$, and

\item $\Card W(k)/H(k)\leq 3^{72 k\phi(2k)}$.
\end{enumerate}

\noindent By Lemma \ref{systole}, $H(k)$ acts freely on $\mathcal{H}(k)$. 
Let $g_k$ be the genus of the closed surface
 $\Sigma(k):=\mathcal{H}(k)/H(k)$ and let
 $\Sys(\Sigma(k))$ be the set of systoles of 
 $\Sigma(k)$. 
Set

\centerline{$\delta=12\sqrt{e^{-\gamma}\ln 3}
$.}

\begin{lemma}\label{delta}
\begin{enumerate}
\item We have $$\lim_{k\to\infty}\,g_k=\infty$$

\item Let $\delta^+> \delta$. For almost all $k\in B$, we have

$$\Card\,\Sys(\Sigma(k))\leq {6\delta^+\over \sqrt{\ln\ln\ln g_k}}\,\,
{g_k\over \sqrt{ \ln g_k}}.$$

 \end{enumerate}
\end{lemma}

\begin{proof}
By Lemma \ref{systole},  $\Sys(\Sigma(k))$ is exactly the set of curves of the tesselation
$\tau_{H(k)}$. Since the systole length of $\Sigma(k)$
tends to $\infty$ as $k$ tends to $\infty$, so is its genus $g_k$, which proves the first assertion.

In the proof of the second assertion, the integer $k$ varies over $B$.
Set  

\centerline{$\delta(k)=3^{72 k\phi(2k)}=3^{144 k\phi(k)}$.}

\noindent By
Theorem \ref{Landau}, we have
 $\ln\delta(k)
\sim 144e^{-\gamma}\ln 3k^2\,/\ln\ln k$, hence

(24.1)\hskip3cm
$ (\ln\ln k)\, \ln\delta(k)\sim \delta^2 k^2$.

\noindent It follows that 

(24.2)\hskip3cm $k>\sqrt{\ln\delta(k)}$, for $k>>0$.

\noindent When $k$ tends to infinity, we have
$\ln\ln\ln \delta(k)\sim \ln\ln\ln \sqrt{\delta(k)}$.
Hence Equation (24.2) implies that

(24.3)\hskip3cm $(\delta^+)^2\ln\ln k>\delta^2\ln\ln\ln \delta(k)$ 
for $k>>0$.

Combining Equations (24.1) and (24.3) we get that

\centerline{$(\delta^+ k)^2>\ln\delta(k)\,\ln\ln\ln \delta(k)$
for $k>>0$,}

\noindent thus we have

(24.4)\hskip3cm
${1\over k}< {\delta^+\over\sqrt{\ln\ln\ln \delta(k)}} \,{1\over \sqrt{\ln\delta(k)}}$, for $k>>0$.

\smallskip
Let $f_0, f_1$ and $f_2$ be the number of vertices, edges and tiles of the tesselation $\tau_{H(k)}$.
Since it is a hexagonal tesselation and each vertex 
has valence $4$, we have
$f_1=3f_2$ and $f_0=f_1/2$. Since
$2(g_k-1)=f_1-f_2-f_0$, we have
$2(g_k-1)=f_2/2=[W(k):H(k)]/2$, hence
we have $g_k\leq \delta(k)$. It follows that

(24.5)\hskip3cm
${1\over k}< {\delta^+\over\sqrt{\ln\ln\ln g_k}} \,{1\over \sqrt{\ln g_k}}$, for $k>>0$.

\noindent
The number  of systoles is 
$f_1/2k=6(g_k-1)/k<6 g_k/k$. It follows from
equation  (24.5) that

(24.6)\hskip3cm
$\Card\,\Syst(\Sigma(k))< {6\delta^+\over\sqrt{\ln\ln\ln g_k}} \,{g_k\over \sqrt{\ln g_k}}$, for $k>>0$.
\end{proof}

\bigskip\noindent
{\it 5.3 The bound for $\mathrm{Fill}(g)$} 

\noindent The following statement is a stronger form 
of the theorem stated in the introduction.

\begin{thm}\label{IM} There exists 
an infinite set $A$ of integers $g\geq 2$ and,
for any $g\in A$,  
a   closed oriented hyperbolic surface $\Sigma_g$ of genus $g$,
endowed with a standard hexagonal tesselation $\tau_g$, satisfying the following assertions

\begin{enumerate}

\item  the set of curves of $\tau_g$
 is the set of systoles of $\Sigma_g$, and

\item we have

\centerline{$\Card\,\Syst(\Sigma_g)\leq {57\over \sqrt{\ln\ln\ln g}}\,\,
{g\over \sqrt{ \ln g}}$.}
\end{enumerate}
\end{thm}

\begin{proof} Let us use the notations
of Subsection 5.2 and set $\delta^+=9.5=57/6$. 

By 
the first assertion of Lemma \ref{delta}, there is an infinite subset $B'\subset B$ such that
 the map $k\in B'\mapsto g_k\in\Z$ is injective,
Set $A:=\{g_k\mid k\in B'\}$ and, for $g\in A$
set $\Sigma_g=\Sigma(k)$ and $\tau_g=\tau_{H(k)}$ where
$k\in B'$ is uniquely defined by $g_k=g$. 

It follows from the second assertion of Lemma \ref{delta} that 

\hskip3cm$\Card\,\Sys(\Sigma_g)\leq {57\over \sqrt{\ln\ln\ln g}}\,\,
{g\over \sqrt{ \ln g}}$, for any $g\in A$.
\end{proof}

\bigskip\noindent
{\it 5.4 Final remark} 

\noindent The constant $57$ in the theorem  can be replaced by any real number
$a>6\delta=56.547\dots$. This constant can be improved in two ways.

First, one can use the fact that the Tits representations lies inside the orthogonal group
$\O_6(K,q)$, where $q$ is the quadratic form defined by $B$. 
Second, the results concerning hexagon tesselations clearly extend to $2p$-gon tesselations, for any $p\geq 3$. Using octogons instead of hexagons  
provides a marginally better bound. 
We have restricted ourselves to this version in order
to keep the paper as elementary as possible.

The paper \cite{APP} and our result suggests that $\Fill(g)$ should be of 
``order of magnitude" $g/(\ln g)^\alpha$ for some $\alpha$ with
$1/2\leq \alpha\leq 1$, but we cannot formulate a precise conjecture at this stage.

\bibliography{SystoleStoryBib}
\bibliographystyle{plain}
\end{document}